\title
{When does a biased graph come from a group labelling?}
\author{
   Matt DeVos\thanks{Department of Mathematics, Simon Fraser University, 8888 University Drive, Burnaby, BC, Canada. Email: mdevos@sfu.ca.
     Supported in part by an NSERC Discovery Grant (Canada).}
\and
   Daryl Funk\thanks{Department of Mathematics, Simon Fraser University, 8888 University Drive, Burnaby, BC, Canada. Email: dfunk@sfu.ca.  Supported in part by an NSERC Postgraduate Scholarship.}
\and
   Irene Pivotto\thanks{School of Mathematics and Statistics, University of Western Australia, 35 Stirling Highway, Crawley, WA, Australia. Email: irene.pivotto@uwa.edu.au. Supported by an Australian Research Council Discovery Project (project number DP110101596).}
}

\date{}

\documentclass[12pt]{article}

\usepackage{graphicx}
\usepackage{amsfonts}
\usepackage{amssymb}
\usepackage{amsmath}
\usepackage{amsthm}
\usepackage[hypcap]{caption}

\usepackage{hyperref} 

\DeclareMathOperator{\rank}{rank}

\begin{document}

\bibliographystyle{plain}
\maketitle
\setcounter{page}{1}
\newtheorem{theorem}{Theorem}[section]
\newtheorem{lemma}[theorem]{Lemma}
\newtheorem{corollary}[theorem]{Corollary}
\newtheorem{proposition}[theorem]{Proposition}
\newtheorem{definition}[theorem]{Definition}
\newtheorem{claim}{Claim}
\newtheorem{conjecture}[theorem]{Conjecture}
\newtheorem{observation}[theorem]{Observation}

\begin{abstract} 
A biased graph consists of a graph $G$ together with a collection of distinguished cycles of $G$, called balanced, with the property that no theta subgraph contains exactly two balanced cycles.  Perhaps the most natural biased graphs on $G$ arise from orienting $G$ and then labelling the edges of $G$ with elements of a group $\Gamma$.  In this case, we may define a biased graph by declaring a cycle to be balanced if the product of the labels on its edges is the identity, with the convention that we take the inverse value for an edge traversed backwards.  Our first result gives a natural topological characterisation of biased graphs arising from group-labellings.  

In the second part of this article, we use this theorem to construct some exceptional biased graphs.  Notably, we prove that for every $m \ge 3$ and $\ell$ there exists a minor-minimal not group-labellable biased graph on $m$ vertices where every pair of vertices is joined by at least $\ell$ edges.  
In particular, this shows that biased graphs are not well-quasi-ordered under minors. 
Finally, we show that these results extend to give infinite sets of excluded minors for certain natural families of frame and lift matroids, and to show that neither are these families well-quasi-ordered under minors.  
\end{abstract}



\section{Introduction} 


Throughout we shall assume that all graphs are finite, but may have loops and parallel edges.  A \emph{theta graph} consists of two distinct vertices $x,y$ and three internally disjoint paths from $x$ to $y$.  
A \emph{biased graph} consists of a pair $(G, \mathcal{B})$ where $G$ is a graph and $\mathcal{B}$ is a collection of cycles, called \emph{balanced}, obeying the \emph{theta property} - that is, there does not exist a theta subgraph of $G$ for which exactly two of the three cycles are balanced.  Cycles not in $\mathcal{B}$ are called \emph{unbalanced}.  We view ordinary graphs as a special case of biased graphs where every cycle is balanced.  

The theory of biased graphs was developed by Zaslavsky (see for example \cite{MR1007712, MR1088626, MR1273951, MR2017726}).  More recently, biased graphs have risen to prominence thanks to the central role they play in the Matroid Minors Project (see, for example \cite{2013arXiv1312.5012G, MR2467829, MR2179649}). 


Perhaps the most natural families  of biased graphs arise from group-labelled graphs (also called gain graphs).  A \emph{group labelling} of a graph $G$ consists of an orientation of the edges of the graph together with a function $\phi : E(G) \rightarrow \Gamma$, where $\Gamma$ is a group (written multiplicatively).  Consider a walk $W$ in the 
underlying graph of $G$ with edge sequence $e_1, e_2, \ldots, e_{\ell}$ and define $\epsilon_i$ using the orientation of $G$ as follows
\[ \epsilon_i = \left\{ \begin{array}{cl}
				1	&	\mbox{if $e_i$ is traversed forward in $W$}	\\
				-1	&	\mbox{if $e_i$ is traversed backward in $W$}
				\end{array} \right. \]
Now we extend $\phi$ by defining 
\[ \phi(W) = \prod_{i=1}^{\ell} \phi(e_i)^{\epsilon_i}.\]

For a group labelling of $G$ with function $\phi$ we define $\mathcal{B}_{\phi}$ to be the set of all cycles $C$ of $G$ for which 
some (and thus every) simple closed walk $W$ around $C$ satisfies $\phi(W) = 1$.  
It is well known that  $(G, \mathcal{B}_{\phi})$ is a biased graph \cite{MR1007712}; such a graph is \emph{$\Gamma$-labelled}.  

We say that a biased graph $(G, \mathcal{B})$ is $\Gamma$-\emph{labellable} if there is a group labelling of $G$ given by $\phi : E(G) \rightarrow \Gamma$ so that $(G, \mathcal{B}_{\phi}) = (G, \mathcal{B})$.  If $(G, \mathcal{B})$ is $\Gamma$-labellable for some group $\Gamma$ then we say it is \emph{group labellable}.  Our first result gives a topological criteria to determine if a biased graph is group labellable.  

\begin{theorem}
\label{main1}
Let $(G, \mathcal{B})$ be a biased graph and construct a 2-cell complex $K$ from $G$ by adding a disc with boundary $C$ for every $C \in \mathcal{B}$.  
Then the following are equivalent.
\begin{enumerate}
\item $(G, \mathcal{B})$ is group labellable.
\item Every cycle $C \not\in \mathcal{B}$ is a non-contractible curve in $K$.
\end{enumerate}
\end{theorem}

There is a natural notion of minor for biased graphs which extends the usual notion for graphs (which we define in Section~\ref{sec:construction}).
For every group $\Gamma$, let $\mathcal{G}_{\Gamma}$ denote the family of all biased graphs which can be $\Gamma$-labelled.  
It is not difficult to check that if $(G,\mathcal{B})$ is $\Gamma$-labellable, then so are its minors \cite{MR1007712}.  
Every $\mathcal{G}_{\Gamma}$  is therefore a proper minor closed class of biased graphs.  
It is natural therefore to ask about its set of excluded minors - \emph{i.e.}\ the minor minimal biased graphs which are not $\Gamma$-labellable.  
Using Theorem \ref{main1} we give a general construction for such biased graphs for infinite groups.  
Our next result is a consequence of this.

\begin{theorem}
\label{complete-construction}
For every $t \ge 3$ and $\ell$ there exists a biased graph $(G,\mathcal{B})$ with the following properties:
\begin{enumerate}
\item $G$ is a graph on $t$ vertices  and every pair of vertices is joined by at least $\ell$ edges.
\item $(G,\mathcal{B})$ is not group-labellable.
\item For every infinite group $\Gamma$, every proper minor of $(G,\mathcal{B})$ is $\Gamma$-labellable.
\end{enumerate}
\end{theorem}

The famous graph minors theorem of Robertson and Seymour says that every proper minor closed class of graphs is characterised by a finite list of excluded minors.  
Theorem \ref{complete-construction} shows that for every infinite group $\Gamma$ the class $\mathcal{G}_\Gamma$ has a rich set of excluded minors.  In particular, we have the following obvious consequence.

\begin{corollary} \label{corollary:inf_ex_min_for_G_Gamma}
For every infinite group $\Gamma$ and every $t \ge 3$ there are infinitely many excluded minors for $\mathcal{G}_{\Gamma}$ with exactly $t$ vertices.
\end{corollary}

For both graphs and biased graphs, there are natural partial orders defined by the rule that a graph $G$ (biased graph $(G, \mathcal{B})$) dominates another graph $H$ (biased graph $(H, \mathcal{C})$) if and only if $H$  ($(H,\mathcal{C})$) is isomorphic to a minor of $G$ ($(G, \mathcal{B})$).  
An equivalent statement of Robertson and Seymour's graph minors theorem is that 
for graphs, this partial order has no infinite antichain.  
In contrast, the above result shows that the partial order for biased graphs has infinite antichains, even with each member on a fixed number of vertices.  

In fact, there are some very easily described infinite antichains of biased graphs.  For instance, let $2C_n$ denote the graph obtained from a cycle of length $n$ by adding an edge in parallel with every existing edge.  Let $\mathcal{B}_n$ consist of two edge disjoint cycles of length $n$ in the graph $2C_n$.  
Then each $(2C_n, \mathcal{B}_n)$ is a biased graph.   
\begin{observation} \label{obs:2C_n_antichain} 
The set $\{ ( 2C_n, \mathcal{B}_n ) \mid n \ge 3 \}$ is an infinite antichain.  
\end{observation}
To see this, note that each of these biased graphs has exactly two balanced cycles, but contracting or deleting an edge gives a biased graph with fewer than two balanced cycles, and this will remain true under further deletions and contractions.  In Section \ref{sec:infinite_antichains} we show that for every infinite group $\Gamma$, all of these biased graphs are contained in $\mathcal{G}_{\Gamma}$.  Even more pathologically, compared to the situation for graphs, is the following result showing that $\mathcal{G}_{\Gamma}$ may contain infinite antichains all of whose members are on a fixed number of vertices.

\begin{theorem}
\label{inf-antichain-gamma}
Let $\Gamma$ be a group and fix $t \ge 3$.  There exists an infinite antichain of $\Gamma$-labelled graphs on $t$ vertices if and only if $\Gamma$ is infinite.
\end{theorem}

For each biased graph $(G, \mathcal{B})$, there are two matroids naturally associated with $(G, \mathcal{B})$, on ground set $E(G)$, the \emph{lift} matroid $L(G,\mathcal{B})$ and \emph{frame} matroid $F(G,\mathcal{B})$.  
These were defined by Zaslavsky in \cite{MR1088626}.  They may be defined in terms of circuits as follows.  
A set $C \subseteq E(G)$ is a circuit of the lift matroid $L(G, \mathcal{B})$ if $C$ is 
  balanced, 
  the union of two unbalanced cycles meeting in at most one vertex, or 
  a theta subgraph containing no balanced cycle.  
A set $C \subseteq E(G)$ is a circuit of the frame matroid $F(G, \mathcal{B})$ if $C$ is 
  balanced, 
  the union of two unbalanced cycles meeting in at most one vertex together with a path connecting them if these cycles are disjoint, or 
  a theta subgraph containing no balanced cycle.  
Minor operations on $(G,\mathcal{B})$ are consistent with their corresponding matroid minor operations on $L(G,\mathcal{B})$ and $F(G,\mathcal{B})$, and each of the classes of lift and frame matroids are closed under minors  \cite{MR1088626}.  

Spikes and swirls are two families of matroids that have been an important source of examples in studies of representability of matroids over fields.  
For each integer $n \geq 3$, a rank $n$ \emph{spike} is obtained by taking $n$ concurrent three-point lines $\{x_i, y_i, z\}$ ($i \in \{1, \ldots, n\}$) freely in $n$-space, then deleting their common point of intersection $z$.  
A rank $n$ \emph{swirl} is obtained by adding a point freely to each 3-point line of the rank $n$ whirl, then deleting those points lying on the intersection of two 3-point lines.  
Zaslavsky \cite{MR2017726} observed that spikes are lift matroids and swirls are frame matroids both coming from biased graphs of the form $(2C_n, \mathcal{B})$ where every cycle in $\mathcal{B}$ is of length $n$.  
The family of biased graphs $(2C_n, \mathcal{B}_n)$ defined above yields both an infinite antichain of spikes and swirls, since in both cases these matroids have exactly two circuit hyperplanes which partition the ground set, but the same is not true of any proper minor.


For every group $\Gamma$, we let $\mathcal{F}_{\Gamma}$ (resp.\ $\mathcal{L}_{\Gamma}$) denote the class of matroids which can be represented as a frame (lift) matroid of a biased graph which is $\Gamma$-labellable.  Each of these is a proper minor closed class of matroids.  In general, a matroid in either of these classes may have many different representations as biased graphs, which complicates the problem of determining excluded minors.  Fortunately, our constructions have essentially unique representations, and this permits us to achieve the following somewhat surprising result (which we prove in Section \ref{sec:Mat_Ex_Minors}).  

\begin{theorem}
\label{matroid-minors}
For every infinite group $\Gamma$ and every $t \ge 3$ the classes $\mathcal{L}_{\Gamma}$ and $\mathcal{F}_{\Gamma}$ have infinitely many excluded minors of rank $t$.  
\end{theorem}

In addition, we prove that for every infinite group $\Gamma$ and every $t \ge 3$ there exist infinite antichains of rank $t$ matroids in both $\mathcal{L}_{\Gamma}$ and $\mathcal{F}_{\Gamma}$.


\section{A Topological Characterisation}

Theorem \ref{main1} consists of statements 1 and 3 of Theorem \ref{main2}, which we prove next.  
For a graph $G$, group labelled by $\phi : E(G) \rightarrow \Gamma$, our basic definitions assign a notion of balance to each cycle.  This notion naturally extends from cycles to closed walks.  For an arbitrary closed walk $W$, we define $W$ to be \emph{balanced} if $\phi(W) = 1$ and call it \emph{unbalanced} otherwise.  

Let $W$ be a closed walk in the biased graph $(G, \mathcal{B}$), let $W'$ be a subwalk of $W$ which is a path from $u$ to $v$ and assume that $C$ is a balanced cycle of $G$ which contains the path $W'$.  Let $W''$ be the path from $u$ to $v$ in $C$ distinct from $W'$ and modify $W$ to a new closed walk $W^*$ by replacing $W'$ by $W''$.  In this case we say that $W^*$ is obtained from $W$ by \emph{rerouting along a balanced cycle}, or simply, by a \emph{balanced rerouting}.  
If $\mathcal{B} = \mathcal{B}_{\phi}$ for a group labelling $\phi$, then since $C$ is balanced, $\phi(W') = \phi(W'')$, so $\phi(W^*) = \phi(W)$.  

\begin{theorem}
\label{main2}
Let $(G, \mathcal{B})$ be a biased graph and let $K$ be the 2-cell complex obtained from $G$ by adding a disc with boundary $C$ for every $C \in \mathcal{B}$.  Then the following are equivalent.
\begin{enumerate}
\item $G$ is group labellable
\item $G$ is $\pi_1(K)$-labellable.
\item Every cycle $C \not\in \mathcal{B}$ is noncontractible in $K$.
\item There does not exist a sequence of closed walks $W_1, \ldots, W_n$ so that 
	each $W_{i+1}$ is obtained from $W_i$ by a balanced rerouting, $W_1$ is a simple walk around an unbalanced cycle and $W_n$ is a simple walk around a balanced cycle.
\end{enumerate}
\end{theorem}

\begin{proof} Trivially (2) implies (1), and our preceding discussion noted that (1) implies (4).   So, to complete the proof it will suffice to show that (3) implies (2), and the negation of (3) implies the negation of (4).  

We may assume that $G$ is a connected graph (as the theorem operates independently on components) and choose a spanning tree $T$.  
Let $(G',\mathcal{B}')$ denote the (one vertex) biased graph obtained from $(G, \mathcal{B})$ by contracting every edge in $E(T)$.  
Let $K'$ denote 
the cell complex obtained from $K$ by identifying $T$ to a single point.  
Since $T$ is contractible, it follows that $\pi_1(K) \cong \pi_1(K')$ (see Proposition 0.17 in \cite{MR1867354}).  

We now apply a standard result to obtain a natural description of the fundamental group of $K'$.  Give $G'$ an arbitrary orientation, and for every edge $e \in E(G')$ let $\gamma_e$ be a variable.  
For every cycle $C \in \mathcal{B}$ choose a simple closed walk around $C$, and let $e_1, \ldots, e_m$ be the sequence of edges of this walk appearing in $E(G')$ (so this closed walk becomes a sequence of loops on the single vertex of $G'$, obtained by removing from the closed walk around $C$ those edges in $T$). 
For $i \in \{1, \ldots, m\}$, define $\epsilon_i$ to be $1$ if $e_i$ is forward in this walk and $-1$ if it is traversed backward.  Now define $\beta_{C}$ to be the word $\gamma_{e_1}^{\epsilon_1} \gamma_{e_2}^{\epsilon_2} \ldots \gamma_{e_n}^{\epsilon_n}$.  Define $\Gamma$ to be the group presented by the generating set $\{ \gamma_{e} \mid e \in E(G') \}$ with the relations given by setting the words in $\{ \beta_{C} \mid C \in \mathcal{B} \}$ to be the identity.  It follows from an application of Van Kampen's theorem (see Section 1.2 in \cite{MR1867354}) that $\Gamma \cong \pi_1(K') \cong \pi_1(K)$ and furthermore, a closed walk $W$ given by the edge sequence $e_1, \ldots, e_m$ with orientations $\epsilon_1, \ldots, \epsilon_m$ will be contractible in $K'$ if and only if the product $\prod_{i=1}^m \gamma_{e_i}^{\epsilon_i}$ is equal to the identity in $\Gamma$.  

Our next step will be to define a $\Gamma$-labelling of the graph $G$ given by $\phi : E(G) \rightarrow \Gamma$.  For an edge $e \in E(T)$, we orient it arbitrarily and assign $\phi(e) = 1$.  For an edge $e \in E(G) \setminus E(T)$ we orient $e$ as it was oriented in $G'$ and then define $\phi(e) = \gamma_e$.  Let $W$ be a closed walk in $G$ and let $W'$ be the corresponding closed walk in $G'$.  Suppose that $W'$ has edge sequence $e_1, \ldots, e_m$ and that $\epsilon_i = 1$ if $e_i$ is forward in $W'$ and $\epsilon_i = -1$ if it is backward.  Now we have
\begin{align*}
\mbox{$W$ is contractible in $K$} 
	 \iff \mbox{$W'$ is contractible in $K'$}	
	 \iff  \prod_{i=1}^m \gamma_{e_i}^{\epsilon_i} = 1 
	\iff \phi(W) = 1.
\end{align*}
Every balanced cycle in $G$ will be contractible in $K$, so we automatically have $\mathcal{B} \subseteq \mathcal{B}_{\phi}$.  If (3) holds, then every cycle $C \not\in \mathcal{B}$ is uncontractible in $K$ and the above equation implies that $\mathcal{B} = \mathcal{B_{\phi}}$ so $(G, \mathcal{B})$ is $\Gamma$-labellable and (2) holds.  On the other hand, if (3) is violated, there is a cycle $C \not\in \mathcal{B}$ which is contractible in $K$, and a simple closed walk $W_1$ around $C$ will satisfy $\phi(W_1) = 1$.  In this case, the group relations in $\Gamma$ which reduce the product of the corresponding edge labels to the identity yield a sequence of closed walks which violate (4). 
\end{proof}

\bigskip

In the preceding theorem it is shown that whenever $(G, \mathcal{B})$ has a group labelling, it has one using the group $\pi_1(K)$.  In fact, the labelling using this group constructed in the proof has a natural extreme property.  If $\phi$ and $\psi$ are two group labellings of $(G, \mathcal{B})$, then by definition we have $\mathcal{B}_{\phi} = \mathcal{B} = \mathcal{B}_{\psi}$ so these group labellings have the same set of balanced cycles.  However, it is quite possible for a closed walk $W$ to satisfy $\phi(W) = 1$ and $\psi(W) \neq 1$.  The group labelling constructed in the above proof has the unique minimal set of balanced closed walks.  That is, any closed walk which is balanced in the group-labelling defined there will also be balanced under any other valid group-labelling.

\section{General Construction}\label{sec:construction}
In this section we use Theorem \ref{main1} to give a general construction of some biased graphs which are minor-minimal subject to being not group labellable. Before we explain this construction we define minors for biased graphs.

For an edge $e \in E(G)$ we \emph{delete} $e$ from $(G,\mathcal{B})$ by deleting $e$ from $G$ and then removing from $\mathcal{B}$ every cycle containing $e$.  
For a balanced loop $e$, the contraction $(G, \mathcal{B}) /e$ is defined as $(G, \mathcal{B}) \setminus e$.  
For a non-loop edge $e$, we \emph{contract} $e$ from $(G, \mathcal{B})$ by contracting $e$ in the graph and then declaring a cycle $C$ to be balanced if either $C \in \mathcal{B}$ or $E(C) \cup \{e\}$ is the edge set of a cycle in $\mathcal{B}$.  
It is straightforward to verify that both deletion and contraction preserve the theta property, so these operations always yield a new biased graph.  A \emph{minor} of $(G, \mathcal{B})$ is any biased graph formed by a sequence of deletions and contractions.  
(Contraction of an unbalanced loop is permitted but defined differently depending upon whether it is the associated lift or frame matroid one is interested in, so that the operations remain consistent with those in the associated matroids.  
Because our special biased graphs have no unbalanced loops and we only ever delete or contract one edge, we never need to perform a contraction of an unbalanced loop.)

\bigskip

\noindent{\bf Construction:}
Let $G$ be a simple graph embedded in the plane which is equipped with a $t$-vertex colouring satisfying the following:
\begin{enumerate}
\item $G$ is a subdivision of a 3-connected graph.
\item Every colour appears exactly once on every face (so every face has size $t$).
\item Every cycle of $G$ of size $\le t$ is the boundary of a face.
\end{enumerate}
Now we form a graph $\widetilde{G}$ from $G$ by identifying each colour class to a single vertex.  Define $\mathcal{B}$ to be the set of all cycles of $\widetilde{G}$ which correspond to boundaries of finite faces of $G$.  
We claim that $(\widetilde{G}, \mathcal{B})$ is a biased graph.  Since every cycle in $\mathcal{B}$ is a Hamiltonian cycle of $\widetilde{G}$, the only way for a theta subgraph of $\widetilde{G}$ to contain two members $C, C'$ of $\mathcal{B}$ would be for this theta subgraph to have two edges in parallel, with $C$ and $C'$ sharing all but this pair of edges.  But then this pair of edges would be a parallel pair in $G$, contradicting the assumption that $G$ is simple.  Thus each theta subgraph of $\widetilde{G}$ contains at most one member of $\mathcal{B}$ and we conclude that $(\widetilde{G}, \mathcal{B})$ is a biased graph.  

\begin{theorem}
\label{general-construction}
The biased graph $(\widetilde{G},\mathcal{B})$ constructed above is not group labellable.  For every edge $e$ and every infinite group $\Gamma$, each of the biased graphs obtained by deleting and contracting $e$ are $\Gamma$-labellable.
\end{theorem}

\begin{proof} Let $K$ be the 2-cell complex obtained from the embedded graph $G$ by removing the infinite face.  Thus $K$ is a disc and its boundary is a cycle $C$.  Now let $\widetilde{K}$ be the 2-cell complex obtained from $K$ by identifying each colour class of vertices to a single point.  The cycle $C$ is a contractible curve in $K$, so it is also a contractible curve in $\widetilde{K}$.  Since $C \not\in \mathcal{B}$, by Theorem \ref{main1}, $(\widetilde{G}, \mathcal{B})$ is not group-labellable.

Now let $e \in E(\widetilde{G})$, and let $\Gamma$ be an infinite group (written multiiplicatively).  We construct a $\Gamma$-labelling of $(\widetilde{G},\mathcal{B}) \setminus e$ and a $\Gamma$-labelling of $(\widetilde{G},\mathcal{B})/e$.  
In preparation for this we choose a useful sequence of group elements.  
Choose $g_0 \in \Gamma \setminus \{1\}$.  For $1 \le k \le |E(G)| + |V(G)|$ choose $g_k \in \Gamma$ so that $g_k$ cannot be expressed as a word of length $\le 3t$ using $g_0, g_0^{-1}, \ldots, g_{k-1}, g_{k-1}^{-1}$.  

\bigskip
\newpage 

\noindent{\it Contraction}

Write $(\widetilde{G}',\mathcal{B}') = ({\widetilde{G}}, \mathcal{B}) /e$.  
Since every cycle in $\mathcal{B}$ is Hamiltonian in $\widetilde{G}$, every such cycle not containing $e$ will form handcuffs upon contracting $e$.  So the only cycles in $\mathcal{B}'$ correspond to finite faces of the planar graph $G$ which contain $e$; thus $|\mathcal{B}'| \le 2$.  
To $\Gamma$-label $\widetilde{G}'$, we label $E(G) \setminus e$; $\widetilde{G}'$ then inherits its labels from $G/e$.  
Let $H$ be the subgraph of $G$ consisting of all its vertices and edges that are on a finite face containing $e$.  It follows from the assumption that $G$ is a subdivision of a 3-connected graph that $H$ must either be a cycle or a theta subgraph (depending on whether $e$ lies on the infinite face or not).  Let $V(H/e) = \{ v_0, \ldots, v_{n} \}$ and let $E(G) \setminus E(H) = \{ e_{n+1}, \ldots, e_m \}$.  To construct the $\Gamma$-labelling, give $G$ an arbitrary orientation, and assign edge labels as follows.  For every edge $f \in E(H/e)$, if $f= v_i v_j$, oriented from $v_i$ to $v_j$, let $\phi(f) = g_i^{-1} g_j$.  For every edge $e_k \in E(G) \setminus E(H)$, define $\phi(e_k) = g_k$.

We claim that $\phi$ realises $\mathcal{B}'$; \emph{i.e.}\ that $\mathcal{B}_{\phi} = \mathcal{B'}$.  To prove this, let $\widetilde{D}$ be an arbitrary cycle in $\widetilde{G}'$.  We show that either $\widetilde{D}$ is in both $\mathcal{B}'$ and $\mathcal{B}_{\phi}$ or $\widetilde{D}$ is in neither.  Define $D$ to be the subgraph of $G$ induced by $E(\widetilde{D})$ (so $D$ is either a cycle or a union of disjoint paths).  First suppose that $\widetilde{D}$ contains an edge $e_k \in E(G) \setminus E(H)$, and choose such an edge for which $k$ is maximum.  Since $e_k \notin H$, we have $\widetilde{D} \not\in \mathcal{B}'$.  If $W$ is a simple closed walk in $\widetilde{G}'$ around $\widetilde{D}$ beginning with $e_k$ in the forward direction, then $\phi(W)$ has the form $g_k$ times a word of length $< 2(t-1) < 3t$ consisting of group elements in $\{g_0, g_0^{-1}, \ldots, g_{k-1}, g_{k-1}^{-1}\}$.  Thus $\phi(W) \neq 1$ and we have $\widetilde{D} \not\in \mathcal{B}_{\phi}$ as desired.  
So now suppose $E(\widetilde{D}) \subseteq E(H)$.  
If $D$ is a cycle in $H/e$, then $\widetilde{D} \in \mathcal{B}'$ and $\widetilde{D} \in \mathcal{B}_{\phi}$ by definition.  
If $D$ is not a cycle in $H/e$, then $\widetilde{D} \not\in \mathcal{B}'$ and we must show that $\widetilde{D} \not\in \mathcal{B}_{\phi}$.  
Let $D_1, \ldots, D_r$ be the components of $D$, let $W$ be a simple closed walk around $\widetilde{D}$ and assume that $W$ encounters each $D_i$ consecutively.  
If the subwalk $W'$ of $W$ traversing $D_h$ begins at $v_i$ and ends at $v_j$, then we have $\phi(W') = g_i^{-1} g_j$.  Therefore, if we choose $k$ to be the largest value so that $v_k$ is an endpoint of one of the paths $D_1, \ldots, D_r$ then $\phi(W)$ may be expressed as a word of length $\leq 2r < 2(t-1) < 3t$ using exactly one copy of $g_k$ or $g_k^{-1}$ with all other terms equal to one of $g_0, g_0^{-1}, \ldots, g_{k-1}, g_{k-1}^{-1}$.  It follows that $\widetilde{D} \not\in \mathcal{B}_{\phi}$ as desired.

\bigskip

\noindent{\it Deletion}

Now let $(\widetilde{G}', \mathcal{B}') = (\widetilde{G}, \mathcal{B}) \setminus e$.
First suppose that $e$ is incident with the infinite face of $G$.  In this case, let $V(G)= \{v_0, \ldots, v_n\}$ and associate each $v_i$ with group element $g_i$.  Orient the edges in $E \setminus e$ arbitrarily, and for every $f \in E \setminus e$ oriented from $v_i$ to $v_j$ define $\phi(f) = g_i^{-1} g_j$.  We claim that $\mathcal{B}_{\phi} = \mathcal{B}'$.  
To prove this (as before) we let $\widetilde{D}$ be an arbitrary cycle in $\widetilde{G}'$ and we let $D$ be the corresponding subgraph of $G \setminus e$.  As before, the graph $D$ must either be a cycle or a union of disjoint paths.  If $D$ is a cycle, then by property 3 of $G$, it must be a face boundary, so $\widetilde{D} \in \mathcal{B'}$ by definition and $\widetilde{D} \in \mathcal{B}_{\phi}$ by construction.  If $D$ is a union of disjoint paths given by $D_1, \ldots, D_r$, then $\widetilde{D} \not\in \mathcal{B}'$ and we must show that $\widetilde{D} \not\in \mathcal{B}_{\phi}$.  As before, choose a closed walk $W$ traversing $\widetilde{D}$ so that it encounters each $D_h$ consecutively.  If the subwalk $W'$ of $W$ traversing $D_h$ starts at $v_i$ and ends at $v_j$, then $\phi(W') = g_i g_j^{-1}$.  So as before, if $k$ is the largest integer so that $v_k$ is an endpoint of one of the paths $D_1, \ldots, D_r$, we find that $\phi(W)$ may be written as a word of length $\leq 2r \leq 2t < 3t$ using only one copy of either $g_k$ or $g_k^{-1}$ with all other terms one of $g_0, g_0^{-1}, \ldots, g_{k-1}, g_{k-1}^{-1}$.  It follows that $\widetilde{D} \not\in \mathcal{B}_{\phi}$ as desired.

Finally suppose that $e$ is not incident with the infinite face and let $R$ be the new face in $G \setminus e$ formed by deleting $e$ from $G$.  Choose a path $P$ in the dual graph of $G \setminus e$ from the infinite face to $R$ and then orient the edges in $E \setminus e$ so that the edges dual to those in $P$ cross the path $P$ consistently (for instance, if $P$ is given a direction, then $E \setminus e$ may be oriented so that each edge dual to one in $P$ crosses $P$ from the left to the right).  Now let $V(G) = \{v_0, \ldots, v_n\}$ and define a $\Gamma$-labelling as follows.  If $f$ is an edge from $v_i$ to $v_j$ and $f$ is not dual to an edge in $P$, let $\phi(e) = g_i^{-1} g_j$; if $e$ is dual to an edge in $P$, let $\phi(e) = g_i^{-1} g_0 g_j$.  Observe that for any closed walk $W$ in $G \setminus e$ we have $\phi(W) = g_0^s$ where $s$ is the number of times the curve $W$ winds around the face $R$.  Following our above procedure, now let $\widetilde{D}$ be a cycle of $\widetilde{G}'$ and let $D$ be the corresponding subgraph of $G \setminus e$.  If $D$ is a cycle, then since its length is at most $t$, it bounds a face in $G \setminus e$ other than $R$.  If this is a finite face, then $\widetilde{D} \in \mathcal{B}'$ and by definition $\widetilde{D} \in \mathcal{B}_{\phi}$.  If this is the infinite face, then $\widetilde{D} \not\in \mathcal{B}'$ and since this face winds around $R$ exactly once we have $\phi(W) = g_0$ or $\phi(W)=g_0^{-1}$, so $\widetilde{D} \not\in \mathcal{B}_{\phi}$.  
Finally, if $D$ is a union of disjoint paths $D_1, \ldots, D_r$ then $\widetilde{D} \not\in \mathcal{B}'$ and we must show $\widetilde{D} \not\in \mathcal{B}_{\phi}$.  
Choose a closed walk $W$ traversing $\widetilde{D}$ encountering each $D_h$ consecutively.  
Let $W = e_1 e_2 \cdots e_s$.  
Then $s \leq t$, and $\phi(W) = \phi(e_1) \phi(e_2) \cdots \phi(e_s)$ is a word of length $\leq 3s$ since each word $\phi(e_i)$ is a word of the form $g_i^{-1} g_j$, $g_i^{-1} g_0 g_j$, or $g_i^{-1} g_0^{-1} g_j$, and so has length at most 3.  
Letting $k$ be the largest value so that $v_k$ is an endpoint of one of the paths $D_1, \ldots, D_r$ we have that $\phi(W)$ may be written as a word of length $\le 3t$ using just one copy of either $g_k$ or $g_k^{-1}$ and all other terms one of $g_0, g_0^{-1}, \ldots, g_{k-1}, g_{k-1}^{-1}$.  As before, this implies that $\phi(W) \neq 1$ so $\widetilde{D} \not\in \mathcal{B}_{\phi}$ as desired.
\end{proof}

\section{Excluded Minors - Biased Graphs}

In this section we prove Theorem \ref{complete-construction}, giving us a large collection of minor-minimal not group labellable biased graphs each of whose underlying simple graph is complete.  
We then construct some families of minor-minimal not group labellable biased graphs each of whose underlying simple graphs is a cycle.  
These results are based upon the general construction from the previous section together with certain families of coloured planar graphs.  We begin by introducing two basic families of coloured planar graphs.  (These colourings are proper.)  

For every positive integer $k$ we define $F_{2k}$ to be the coloured planar graph given as follows.  Begin with a cycle of length $2k$ embedded in the plane in which vertices are alternately coloured $0$ and $1$.  Then add two additional vertices, one in each face, each adjacent to all vertices on this cycle and each of colour $a$ (Figure \ref{fig:f_2k}).  
\begin{figure}[htbp]
  \centering
  \includegraphics[scale=0.75]{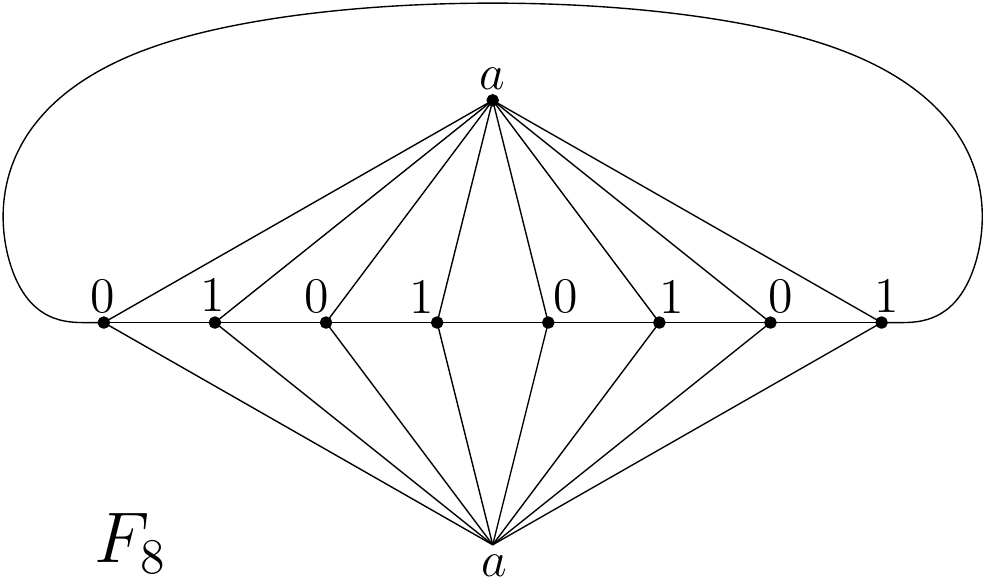}
  \caption{}
  \label{fig:f_2k}
\end{figure}

For every positive integer $k$ we define $H_{2k}$ to be the planar graph constructed as follows.  Begin with $2k$ nested 8-cycles embedded in the plane, each joined to the previous and the next by a perfect matching.  Colour this portion of the graph by colouring the innermost cycle $b$, $0$, $b$, $1$, $b$, $0$, $b$, $1$, and extend this colouring so that every 4-cycle (of the present graph) contains exactly one vertex of each of the colours $\{a,b,0,1\}$ (this extension is unique).  Finally, add a vertex $v_1$ in the inner 8-cycle of colour $a$ joined to all vertices on this cycle not of colour $b$ and similarly, add a vertex $v_2$ in the infinite face coloured $b$ and adjacent to all vertices not of colour $a$ on this face (Figure \ref{fig:h_2k}).  
\begin{figure}[htbp]
  \centering
  \includegraphics[scale=0.9]{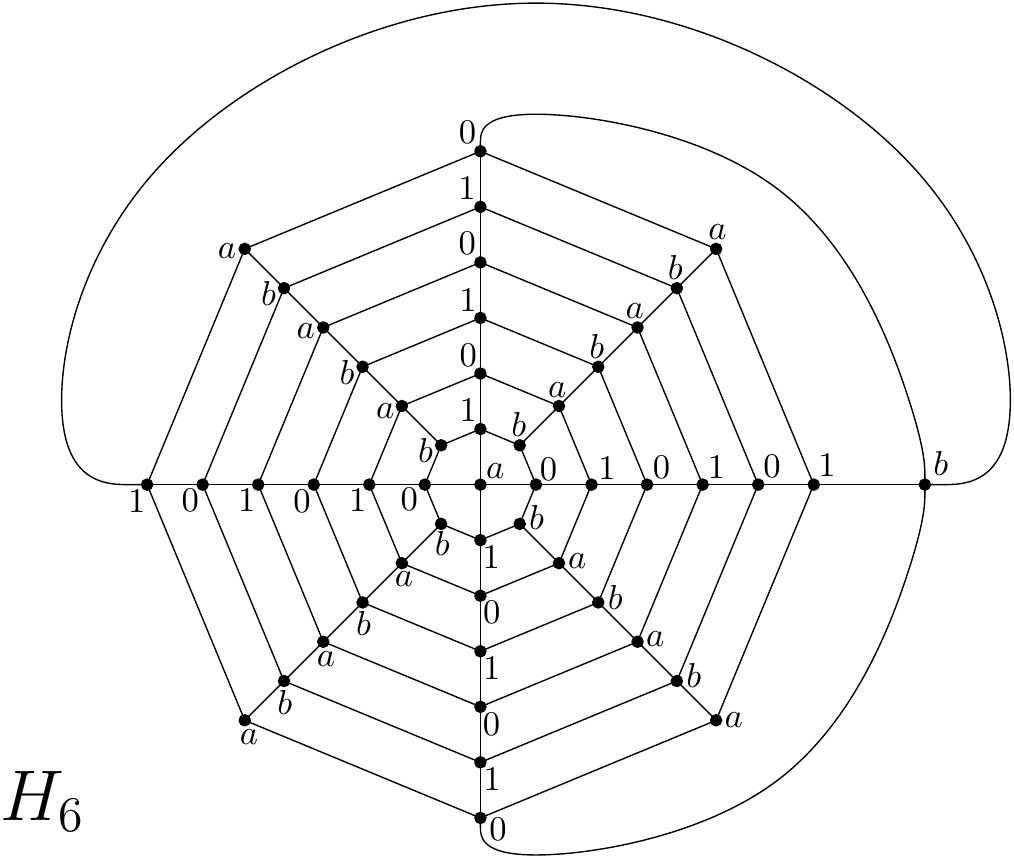}
  \caption{}
  \label{fig:h_2k}
\end{figure}
Next we use these to construct some useful families of coloured planar graphs.  

\begin{lemma}
\label{coloured-planar}
For every $t \ge 3$ and $\ell$ there exists a $t$-coloured planar graph with the following properties:
\begin{enumerate}
\item $G$ is a subdivision of a 3-connected graph.
\item Every colour appears exactly once on every face (so every face has size $t$).
\item Every cycle of $G$ of size $\le t$ is the boundary of a face.
\item Every pair of distinct colours appear on opposite ends of at least $\ell$ edges.  
\end{enumerate}
\end{lemma}

\begin{proof} 
We split into cases depending on the parity of $t$.
\bigskip

\noindent
\textbf{Case 1:} $t$ odd.

\smallskip
For $t=3$ the coloured graphs $F_{2k}$ with $k \geq \min\{\ell/2, 2\}$ have properties 1-4.  
In general, we choose $s$ so that $t = 2s+1$, use the colour set $\{a\} \cup \{1,2,\ldots,2s\}$, and modify $F_{2k}$, taking $k$ as large as necessary to achieve what is required in each step.  
Begin by choosing a sequence $x_1,x_2,\ldots,x_{2k}$ of elements from $\{1,2,\ldots,2s\}$ with the following properties:
\begin{enumerate}
	\item[(i)] every $x_i$ has the same parity as $i$,
	\item[(ii)] every pair of numbers in $\{1,2,\ldots,2s\}$ with differing parities appear consecutively in this sequence at least $\ell$ times, and
	\item[(iii)] every element in $\{1,2,\ldots,2s\}$ appears at least $\ell(s-1)$ times in the sequence.
\end{enumerate}
Now modify the colouring of the graph $F_{2k}$ by replacing the sequence of $0$ and $1$ colours by $x_1, \ldots, x_{2k}$.  Next, for every edge with one end of colour $a$ and the other end an odd (even) colour $i$ we subdivide this edge $s-1$ times and give these new vertices distinct odd (even) colours in $\{1,\ldots,2s\} \setminus \{i\}$.  
We choose this assignment of colours with some extra restrictions, which we explain next.
Let $v_1$ and $v_2$ be the two vertices of colour $a$.  
For every edge $v_1w$ in $F_{2k}$, where $w$ is coloured $x_i$, assign colour $x_i+2$ (modulo $2s$) to the neighbour of $v_1$ in the subdivided edge $v_1w$.
With this choice we have that $v_1$ has at least $\ell$ neighbours of each colour $\{1,\ldots,2s\}$.  
Finally we ensure that every pair of distinct colours of the same parity appear on opposite ends of at least $\ell$ edges by enforcing the following restriction on the choice of colouring of the vertices sharing a face with $v_2$. 
We do the following for every choice of  $j \in \{1,\ldots,2s\}$: let $w_1,\ldots,w_{\ell(s-1)}$ be a set of $\ell(s-1)$ vertices coloured $j$ in $F_{2k}$.  
Let $u_i$ be the degree-2 neighbour of $w_i$ in the subdivided $w_iv_2$ edge. 
For every $n \in \{1,\ldots,s-1\}$ assign colour $j+2n$ to vertices $u_{(n-1)\ell+1}, \ldots, u_{n\ell}$.
The resulting coloured planar graph then has the desired properties.  

\bigskip

\noindent
\textbf{Case 2:} $t$ even.

\smallskip

For $t=4$ the coloured graphs $H_{2k}$ with $k \geq \min\{\ell/4, 2\}$ satisfy properties 1-4.  In general, we choose $s$ so that $t=2s+2$, use the colour set $\{a,b\} \cup \{1,2,\ldots,2s\}$, and modify $H_{2k}$, taking $k$ as large as necessary to achieve what is required in each step.  
Begin by choosing a sequence $x_1,x_2,\ldots,x_{2k}$ of elements from $\{1,2,\ldots,2s\}$ with the following properties:
\begin{enumerate}
	\item[(i)] every $x_i$ has the same parity as $i$,
	\item[(ii)] every pair of numbers in $\{1,2,\ldots,2s\}$ with differing parities appear consecutively in this sequence at least $\ell$ times, and
	\item[(iii)] every element in $\{1,2,\ldots,2s\}$ appears at least $\ell s$ times in the sequence.
\end{enumerate}
Now consider the coloured graph $H_{2k}$.  
Let $P_1, P_3$ be the paths of length $2k-1$ that are coloured alternately $0$ and $1$ beginning with a vertex incident to $v_1$ coloured $1$, and let $P_2, P_4$ be the paths of length $2k-1$ that are coloured alternatively $0$ and $1$ beginning with a vertex coloured $0$ incident to $v_1$.  
Modify the colouring of $H_{2k}$ by replacing the colours along each of $P_1$ and $P_3$ with the sequence of colours $x_1, \ldots, x_{2k}$ (starting at the vertex coloured 1), and replacing the colours along each of $P_2$ and $P_4$ with the sequence of colours $x_2, \ldots, x_{2k}, x_1$ (starting at the vertex coloured 0).  
Note that in this manner we have replaced each vertex previously coloured 0 with an even colour, and each vertex previously coloured 1 with an odd colour.  
Now we modify the graph by the following procedure.  
Aside from the four edges incident with the central vertex $v_1$ (coloured $a$) and the four edges incident with outer vertex $v_2$ (coloured $b$), for every other edge 
\begin{itemize} 
\item  $ai$ or $bi$ with $i$ even: subdivide the edge $s-1$ times and give each new vertex a distinct even colour from $\{1, \ldots, 2s\} \setminus \{i\}$; 
\item  $ai$ or $bi$ with $i$ odd: subdivide the edge $s-1$ times and give each new vertex a distinct odd colour from $\{1, \ldots, 2s\} \setminus \{i\}$. 
\end{itemize}
These subdivisions ensure that every colour appears exactly once on every face.  

Similarly to the previous case, we choose this assignment of colours with some extra restrictions to ensure that for every pair of colour classes there are at least $\ell$ edges joining vertices of different colours.  We now describe these restrictions.  
To help with bookkeeping, we partition the set of all pairs of colour classes into eight types: even-even, odd-odd, even-odd, $a$-even, $a$-odd, $b$-even, $b$-odd, and $a$-$b$ (where each pair of colour classes belongs to the obvious type described by its name).  
Property (ii) of our chosen sequence $x_1, x_2, \ldots, x_{2k}$ ensures that we have at least $\ell$ edges between all even-odd pairs of colour classes.  
Our coloured graph $H_{2k}$ has $4(2k-1)$ edges with one endpoint coloured $a$ and the other endpoint coloured $b$.  These edges remain in our modified graph; since $k$ is taken large enough to accommodate the sequence $x_1, x_2, \ldots, x_{2k}$ required by property (ii), we certainly have at least $\ell$ edges with one endpoint coloured $a$ and the other coloured $b$.  
We ensure that this also holds for all remaining pairs of colour classes by colouring the new vertices on the subdivided edges $ai$ and $bi$ as follows.  
The $2k$ subdivided edges $ai$ with $i$ in $P_1$ have $i$ even: colour the new vertices on these subdivided edges so that there are at least $\ell$ edges with one endpoint of colour $a$ and the other of colour $i$ for each even $i \in \{1, \ldots, 2s\}$.  
The $2k$ subdivided edges $bi$ with $i$ in $P_1$ have $i$ odd: colour the new vertices on these subdivided edges so that there are at least $\ell$ edges with one endpoint of colour $b$ and the other of colour $i$ for each odd $i \in \{1, \ldots, 2s\}$.  
In this way we ensure that there are at least $\ell$ edges between all $a$-even and at least $\ell$ edges between all $b$-odd pairs of colour classes.  
The subdivided edges $ai$ with $i$ in $P_2$ have $i$ odd; subdivided edges $bi$ with $i$ in $P_2$ have $i$ even.  
Colouring the new vertices on these subdivided edges so that there are at least $\ell$ edges with one endpoint of colour $a$ and the other of colour $i$ for each odd $i \in \{1, \ldots, 2s\}$, and at least $\ell$ edges with one endpoint of colour $b$ and other other of colour $i$ for each even $i \in \{1, \ldots, 2s\}$ ensures that there are at least $\ell$ edges between all $a$-odd and all $b$-even pairs of colour classes.  
Remaining are pairs of colour classes of types even-even and odd-odd.  
There are $4k$ subdivided edges of the forms $ai$, $bi$ with $i$ in $P_3$ or $P_4$: colouring these new vertices so that every pair of integers in $\{1, \ldots, 2s\}$ of the same parity appear as endpoints of at least $\ell$ edges, we ensure that there are at least $\ell$ edges between all even-even and all odd-odd pairs of colour classes.  
Keeping in mind that we may take $k$ as large as necessary, this colouring is clearly possible.  
%
The resulting coloured graph now has the desired properties.
\end{proof}

With this, we can easily prove our main result for this section.

\bigskip

\begin{proof}[Proof of Theorem \ref{complete-construction}:]
This follows from Theorem \ref{general-construction} and Lemma \ref{coloured-planar}.
\end{proof}
\bigskip

Our next theorem gives constructions for families of minor-minimal not group labellable biased graphs each of whose underlying simple graph is a cycle.  

\begin{theorem}
For every $k \ge 2$ and for $t=3$ and every $t \ge 5$ there exists a biased graph $(G, \mathcal{B})$ with the following properties:
\begin{itemize}
\item The underlying simple graph of $G$ is $C_t$.
\item If $u,v$ are adjacent vertices they are joined by exactly $2k$ edges.  
\item $(G, \mathcal{B})$ is not group-labellable.
\item For every infinite group $\Gamma$, every proper minor of $(G, \mathcal{B})$ is $\Gamma$-labellable.
\end{itemize}
\end{theorem}

\begin{proof} As in the previous theorem we will construct certain coloured planar graphs and then call upon Theorem \ref{general-construction}.  The graphs we construct have $t$-colourings using the colours $\{0,1,\ldots,t-1\}$ with the following properties:
\begin{itemize}
\item On each face the cyclic ordering of colours is given by either $0,1,\ldots,t-1$ or its reverse.
\item There are exactly $4k$ faces.
\end{itemize} 
Note that the above two properties guarantee that the graph $\widetilde{G}$ obtained from the identification process in our construction will satisfy the first and second properties of the theorem.  When $t=3$ we may obtain such a graph from $F_{2k}$ by changing the two vertices coloured $a$ to colour $2$.  So, we may assume $t \ge 5$.  When $(t,k) \in \{(5,2), (8,2)\}$ the graphs depicted in Figure \ref{fig:special_oct} satisfy the desired properties.

\begin{figure}[h]
  \centering
  \includegraphics[scale=0.9]{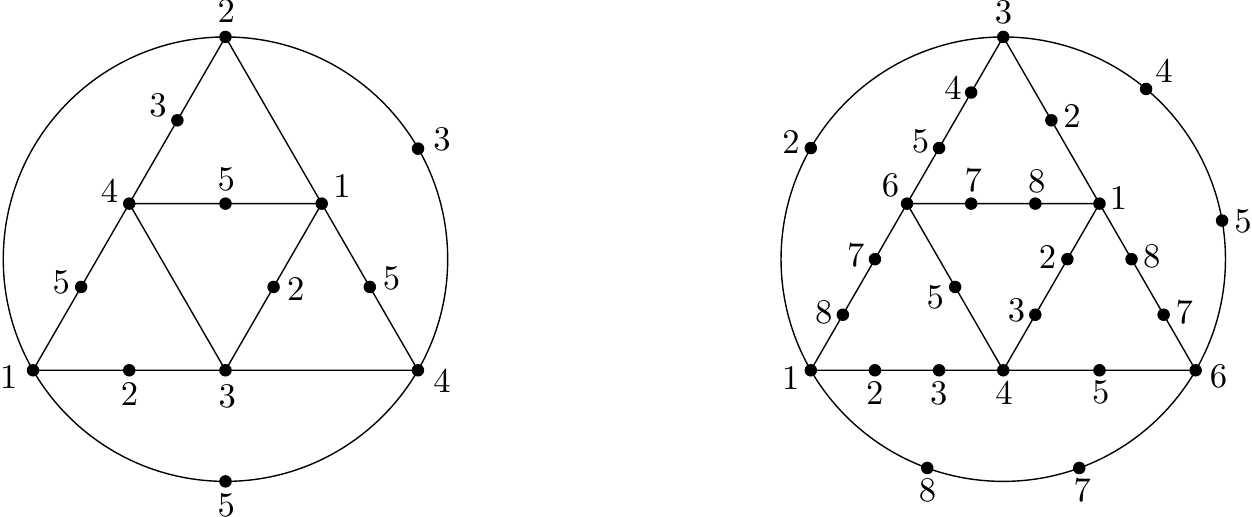}
  \caption{}
  \label{fig:special_oct}
\end{figure}

Thus, we may assume $(t,k) \notin \{(5,2), (8,2)\}$.  Let $t = 3s+p+q$ where $0 \le p \le q \le 1$.  Modify $F_{2k}$ by changing every vertex of colour $a$ to colour $s+p$ and every vertex of colour $1$ to colour $2s+p+q$.  Now subdivide every edge with ends of colours $0$ and $s+p$ exactly $s+p-1$ times, every edge with ends of colours $s+p$ and $2s+p+q$ exactly $s+q-1$ times and every edge with ends of colours $0$ and $2s+p+q$ exactly $s-1$ times.  Now we may colour the vertices of degree two so that around every face they are ordered cyclicly in either clockwise or counterclockwise direction as $0,1,\ldots,3s+p+q-1$.  

\begin{figure}[h]
  \centering
  \includegraphics[scale=1]{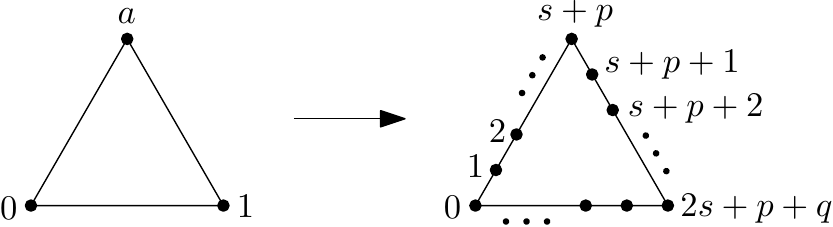}
  \caption{}
  \label{fig:triangle-change}
\end{figure}

In this case each triangle is subdivided as in Figure \ref{fig:triangle-change}.  It follows immediately from our construction that the graphs we have constructed are subdivisions of 3-connected planar graphs with exactly $4k$ faces all coloured as in the figure.  So to complete the proof we need only verify that these graphs have the property that every cycle of size $\le t$ is the boundary of a face.  Observe that every cycle that is not a facial boundary contains at least four vertices of degree $\ge 3$ so will have total length at least $4 \lfloor \frac{t}{3} \rfloor$, which is greater than $t$ for all $t \ge 6$ except for $t=8$.  In the cases when $t=5$ and $t=8$ we have the additional assumption $k \ge 3$ and it is easy to check that any cycle that is not a facial boundary has length at least $6$ if $t=5$ and at least $9$ if $t=8$. So again here we have the desired property.  
\end{proof}

\section{Excluded Minors - Matroids} \label{sec:Mat_Ex_Minors} 

In this section we will call upon our prior results to construct some excluded minors for families of frame and lift matroids.  
We have already shown numerous families of biased graphs which are minor minimal subject to not being $\Gamma$-labellable.  However, this does not immediately give us excluded minors for the classes $\mathcal{L}_{\Gamma}$ and $\mathcal{F}_{\Gamma}$ since there might exist two biased graphs with the same frame (lift) matroid where one is $\Gamma$-labellable and the other is not.  In order to show that we do have excluded minors for these classes, we need to handle this issue of non-unique representations.  

To assist in this exploration, we begin by looking at biased graphs which have associated matroids isomorphic to $U_{2,m}$.  Define $K_2^m$ to be a two-vertex graph consisting of $m$ edges in parallel, let $K_2^{m+}$ be a graph obtained from $K_2^{m}$ by adding a single loop edge, and let $K_2^{m++}$ be a graph obtained from $K_2^{m+}$ by adding another loop not adjacent to the first loop.  

\begin{observation}
For a biased graph $(G,\mathcal{B})$ and $m \ge 4$ we have:
\begin{enumerate}
\item $L(G,\mathcal{B}) \cong U_{2,m}$ if and only if $(G,\mathcal{B})$ is isomorphic to $(K_2^m,\emptyset)$ or $(K_2^{(m-1)+},\emptyset)$.  
\item $F(G,\mathcal{B}) \cong U_{2,m}$ if and only if $(G,\mathcal{B})$ is isomorphic to $(K_2^m,\emptyset)$,  $(K_2^{(m-1)+},\emptyset)$, or 
$K_2^{(m-2)++},\emptyset)$.
\end{enumerate}
\end{observation}

Let us call a biased graph $(G, \mathcal{B})$ \emph{lift-unique} (resp.\ \emph{frame-unique}) if the only biased graphs with lift (resp.\ frame) matroid isomorphic to $L(G, \mathcal{B})$ (resp.\ $F(G, \mathcal{B})$) are obtained from $(G, \mathcal{B})$ by renaming the vertices.  

\begin{lemma}
\label{lift-frame-unique}
Let $(G,\mathcal{B})$ be a loopless biased graph on $n \ge 3$ vertices for which every pair of vertices are joined by at least four edges, and all cycles of length two are unbalanced.  Then $(G,\mathcal{B})$ is both frame-unique and lift-unique.
\end{lemma}

\begin{proof} We begin by considering $F(G, \mathcal{B})$.  Let $E = E(G)$ and define a relation $\sim$ on $E$ by the rule that $e \sim f$ if there exists a restriction of the frame matroid isomorphic to $U_{2,4}$ which contains both $e$ and $f$.  It follows easily from the description of $(G, \mathcal{B})$ that $\sim$ is an equivalence relation and its equivalence classes are precisely the parallel classes of $G$, which we denote by $E_1, E_2, \ldots, E_{ {n \choose 2} }$.  

Suppose that $(G', \mathcal{B}')$ is another biased graph on the same edge set with the same frame matroid; {\it i.e.}, $F(G', \mathcal{B}') \cong F(G, \mathcal{B})$.  If $|E_i| = m$ then the restriction of our matroid to $E_i$ is isomorphic to $U_{2,m}$ and thus in the graph $G'$, the edges in $E_i$ induce a two vertex subgraph isomorphic to one of $K_2^m$, $K_2^{(m-1)+}$ or $K_2^{(m-2)++}$.  It follows from the fact that $\sim$ is an equivalence relation that for $i \neq j$ the edge sets $E_i$ and $E_j$ induce graphs on distinct two-vertex sets.  Next suppose (for a contradiction) that there exists a loop edge $e$ in $G'$ incident with the vertex $v$.  Let $e',e''$ be non-loop edges incident with $v$ which are not in parallel.  Then $e' \in E_i$ and $e'' \in E_j$ for some $i \neq j$. Therefore $e$ is in both $E_i$ and $E_j$, a contradiction.
Thus the graph $G'$ is loopless, and $E_1, \ldots, E_{ {n \choose 2} }$ are also its parallel classes.  

Let $e,f,g$ be three edges which form a triangle in $G$ and let $e'$ be parallel with $e$.  Then one of $\{e,f,g\}$, $\{e',f,g\}$, $\{e,e',f,g\}$ is a circuit in $F(G, \mathcal{B})$.  It follows from this, and the fact that $G'$ is loopless with the same parallel classes as $G$, that the edges $e,f,g$ must also form a triangle in $G'$.  
In particular, this implies that two edges $e,f$ are adjacent in $G$ if and only if they are adjacent in $G'$.  
Therefore, the line graphs of $G$ and $G'$ are isomorphic.  For $n \ge 5$ the maximum cliques in the line graph of $K_n$ correspond precisely to sets of edges incident with a common vertex, and it follows that for $n \ge 5$ the biased graph $(G',\mathcal{B}')$ may be obtained from $(G,\mathcal{B})$ by renaming the vertices.  For $n=3$ there is also nothing left to prove, so we are left with the case $n=4$.  The maximum cliques of the line graph of $K_4$ are given by either triangles or sets of edges incident with a common vertex.  Since three edges form a triangle in $G$ if and only if they form a triangle in $G'$ we conclude that again in this case, the biased graph $(G',\mathcal{B}')$ may be obtained from $(G,\mathcal{B})$ by renaming the vertices.  We conclude that $(G, \mathcal{B})$ is frame-unique.  

For lift matroids the same proof applies with the only difference being that the two vertex subgraph induced by $E_i$ cannot be $K_2^{(m-2)++}$.  
\end{proof}

\bigskip

\begin{proof}[Proof of Theorem \ref{matroid-minors}:] 
Fix $t \ge 3$ and let $\Gamma$ be an infinite group.  
By Theorem \ref{complete-construction} we may choose an infinite set of biased graphs on $t$ vertices $\{ (G_i, \mathcal{B}_i) \mid i \in \{ 1,2,\ldots \} \}$ with $|E(G_{i+1})| > |E(G_i)|$ so that every pair of vertices is joined by at least 4 edges in every $G_i$, every $(G_i, \mathcal{B}_i)$ is not $\Gamma$-labellable, and every proper minor of $(G_i, \mathcal{B}_i)$ is $\Gamma$-labellable.  
Moreover, by the constructions used in the proof of Theorem \ref{complete-construction} we may assume each $(G_i, \mathcal{B}_i)$ is loopless.  
By the previous lemma, each $(G_i, \mathcal{B}_i)$ is both frame-unique and lift-unique.  
We conclude $F(G_i,\mathcal{B}_i) \not\in \mathcal{F}_{\Gamma}$ and $L(G_i,\mathcal{B}_i) \not\in \mathcal{L}_{\Gamma}$.  
Since none of the graphs $(G_i, \mathcal{B}_i)$ have an unbalanced loop, in each of them any single minor operation agrees with the corresponding operation in $L(G_i, \mathcal{B}_i)$ and in $F(G_i, \mathcal{B}_i)$.  
Since $\rank(F(G, \mathcal{B})) = \rank(L(G, \mathcal{B})) = |V(G)|$, it follows that the lift (resp.\ frame) matroid of every $(G_i, \mathcal{B}_i)$ is an excluded minor of rank $t$ for the class $\mathcal{L}_{\Gamma}$ (resp.\ $\mathcal{F}_{\Gamma}$).  
\end{proof}

\bigskip

\section{Infinite Antichains} \label{sec:infinite_antichains}

In the 
proof of Theorem \ref{matroid-minors}, we constructed infinite antichains of biased graphs on a bounded number of vertices by finding biased graphs which are minor minimal subject to not being group labellable.  In this section we prove that for every infinite group $\Gamma$, there also exist infinite antichains of 
$\Gamma$-labellable graphs.  
We also show that these results extend to give infinite antichains of bounded rank in the families of matroids $\mathcal{L}_{\Gamma}$ and $\mathcal{F}_{\Gamma}$.  We begin by noting that the biased graphs $( 2C_n, \mathcal{B}_n )$ of Observation \ref{obs:2C_n_antichain} are group labellable.  

\begin{observation} \label{obs:2C_n_grouplabellable} 
For every infinite group $\Gamma$ and every $n \ge 2$ the biased graph $(2C_n, \mathcal{B}_n)$ is $\Gamma$-labellable.
\end{observation}

\begin{proof}
Orient the edges so that each of the two balanced cycles is a directed cycle, and label all edges in the first balanced cycle with $1$.  Let $e_1, \ldots, e_n$ be the edges of the second balanced cycle in order.  Now choose a sequence of group elements $g_1, \ldots, g_{n-1}$ so that no subsequence of these elements has product equal to $1$ (this may be done greedily).  Assign $e_i$ the label $g_i$ for $1 \le i \le n-1$ and assign $e_n$ the label $g_{n-1}^{-1} \ldots g_1^{-1}$.  
This $\Gamma$-labelling realises $\mathcal{B}_n$.   
\end{proof}

Together, Observations \ref{obs:2C_n_antichain} and \ref{obs:2C_n_grouplabellable} exhibit, for every infinite group $\Gamma$, an infinite antichain of biased graphs in $\mathcal{G}_{\Gamma}$.  We now show (using an argument very similar to that in the proof of Theorem \ref{general-construction}) that there are also such antichains having all members on a bounded number of vertices. 

\begin{lemma}
\label{inf-anti-bias}
For every infinite group $\Gamma$ and every $t \ge 3$ there exists an infinite antichain 
of $\Gamma$-labelled graphs 
on $t$ vertices.
\end{lemma}

\begin{proof}
Apply Lemma \ref{coloured-planar} to choose an infinite family of $t$-coloured planar graphs $\{ G_1, G_2, \ldots \}$ each of which has a distinct number of edges.  Now for every $k$, let $\widetilde{G}_k$ be the graph obtained from $G_k$ by identifying each colour class to a single vertex.  Define 
$\mathcal{B}_k$ to be the set of cycles which are faces of the planar embedding of $G_k$.  

First we prove that every $(\widetilde{G}_k, \mathcal{B}_k)$ is in $\mathcal{G}_{\Gamma}$.  To this end, let $V(G) = \{v_1, \ldots, v_n\}$ and choose a sequence of group elements $g_1, \ldots, g_n$ with the property that each $g_i$ cannot be represented as a product of distinct elements from the set $\{ g_1, g_1^{-1}, \ldots, g_i, g_i^{-1} \}$ (in any order).  Now orient the edges of $G_k$ and thus $\widetilde{G}_k$ arbitrarily, and for every edge $e$ from $v_i$ to $v_j$ define $\phi(e) = g_i^{-1} g_j$.  We claim $\mathcal{B}_{\phi} = \mathcal{B}_k$.  To see this, let $C$ be an arbitrary cycle in $\widetilde{G}_k$.  If $C$ is also a cycle in $G_k$, then it must bound a face, so we have $C \in \mathcal{B}_k$ and, by our construction $C \in \mathcal{B}_{\phi}$.  Otherwise, the set of edges $E(C)$ forms a collection of paths in $G_k$, say $D_1, \ldots, D_r$.   Choose a closed walk $W$ around the cycle $C$ in $\widetilde{G}_k$ and assume that $W$ encounters each $D_i$ consecutively.  It follows from our construction that $\phi(W)$ may be expressed as a product of distinct group elements from $S = \{ g_i \mid \mbox{$v_i$ is an end of some $D_j$} \}$ together with $S^{-1}$.  It follows from our choice of group elements that this product is not the identity.  Hence $C \not\in \mathcal{B}_k$ and $C \not\in \mathcal{B}_{\phi}$ as desired.

It remains to prove that $\{ (\widetilde{G}_k, \mathcal{B}_k) \mid k \in \mathbb{N} \}$ is an antichain.  Suppose that $(\widetilde{G}_i, \mathcal{B}_i)$ contains a biased graph isomorphic to $(\widetilde{G}_j, \mathcal{B}_j)$ as a minor and $i \not= j$.  Since these graphs have the same number of vertices, it must be that 
$(\widetilde{G}_j, \mathcal{B}_j)$ is isomorphic to $(\widetilde{G}_i, \mathcal{B}_i) \setminus R$ for some nonempty set of edges $R$.  Choose an edge $e \in E(\widetilde{G}) \setminus R$ that lies on a common face with an edge in $R$.  Edge $e$ will be in at most one balanced cycle in $(\widetilde{G}_i, \mathcal{B}_i) \setminus R$, but every edge in $(\widetilde{G}_j, \mathcal{B}_j)$ is contained in exactly two balanced cycles, a contradiction.  
\end{proof}

\begin{lemma}
Let $\Gamma$ be a finite group and $t \geq 3$.  There is no infinite antichain of $\Gamma$-labelled graphs on $t$ vertices.  
\end{lemma}

\begin{proof}
Let $G_1, G_2, \ldots$ be an infinite sequence of graphs on the vertex set $\{1,2,\ldots, t\}$ and without loss of generality assume that every edge with ends $i,j$ with $i < j$ is oriented from $i$ to $j$.  For every $k$ let $\phi_k : E(G_k) \rightarrow \Gamma$ be a function.  Let $\Gamma = \{g_1, \ldots, g_{\ell} \}$ and proceed as follows.  For every $G_k$ consider the number of edges between $1$ and $2$ with label $g_1$.  This is an infinite sequence of nonnegative integers, so it has an infinite non-decreasing subsequence.  Now restrict the original sequence of graphs to the corresponding subsequence.  Continuing in this manner for each group element, and then repeating this process for every pair of vertices yields an infinite sequence of group-labelled graphs each contained in the next.  
\end{proof}

\noindent{\it Proof of Theorem \ref{inf-antichain-gamma}:} 
This is an immediate consequence of the previous two lemmas.

\bigskip

Turning our attention to matroids, Lemma \ref{lift-frame-unique} shows that the biased graphs used in the construction of Lemma \ref{inf-anti-bias} are both frame- and lift-unique provided they each have at least four edges between each pair of vertices.  This immediately gives the following corollary.

\begin{corollary} For every infinite group $\Gamma$ and every $t \ge 3$, there exist infinite antichains of rank $t$ matroids in both $\mathcal{L}_{\Gamma}$ and $\mathcal{F}_{\Gamma}$. 
\end{corollary}

\bibliography{Matroids1} 

\end{document}